\documentclass{amsart}

\usepackage{amsmath,amssymb,amsthm,graphicx,pstricks,amscd,amsfonts,latexsym}
\usepackage[all]{xy}

\usepackage{lscape}
\usepackage{hyperref}
\usepackage{tikz}

\usetikzlibrary{shapes}
\usetikzlibrary{decorations.pathreplacing,decorations.markings} 

\usepackage{tikz,tikz-cd}
\usetikzlibrary{snakes}
\usetikzlibrary{shapes.geometric,positioning}
\usetikzlibrary{arrows,decorations.pathmorphing,decorations.pathreplacing}

\theoremstyle{plain} 
\newtheorem{theorem}{Theorem}[section]
\newtheorem{lemma}[theorem]{Lemma}
\newtheorem{proposition}[theorem]{Proposition}
\newtheorem{corollary}[theorem]{Corollary}

\theoremstyle{definition} 
\newtheorem{example}[theorem]{Example}

\newtheorem{remark}[theorem]{Remark}

\newcommand{\End}[1]{\operatorname{\rm End}_{#1}}

\newcommand{\Hom}[1]{\operatorname{{\rm Hom}}_{#1}}

\newcommand{\MOD}{\mbox{{\rm mod \!}}}

\begin{document}
\title{$\tau$-tilting finite gentle algebras are representation-finite}

\author{Pierre-Guy Plamondon}
\address{Laboratoire de Math\'ematiques d'Orsay, Universit\'e Paris-Sud, CNRS, Universit\'e Paris-Saclay, 91405 Orsay, France}
\email{pierre-guy.plamondon@math.u-psud.fr}

\thanks{The author is supported by the French ANR grant SC3A (ANR-15-CE40-0004-01) and by a PEPS ``Jeune chercheuse, jeune chercheur'' grant.}  

\date{\today}


\begin{abstract}
 We show that a gentle algebra over a field is~$\tau$-tilting finite if and only if it is representation-finite.
 The proof relies on the ``brick-$\tau$-tilting correspondence'' of Demonet--Iyama--Jasso and on a combinatorial analysis.
 
\end{abstract}

\maketitle

\tableofcontents

\section{Introduction and main result}

The theory of~$\tau$-tilting was introduced in \cite{AdachiIyamaReiten} as a far-reaching generalization of classical tilting theory for finite-dimensional associative algebras.
One of the main classes of objects in the theory is that of \emph{$\tau$-rigid modules}: a module~$M$ over an algebra~$\Lambda$ is~$\tau$-rigid if the space of morphisms~$\Hom{\Lambda}(M, \tau M)$ vanishes,
where~$\tau$ is the Auslander--Reiten translation.
In \cite{DemonetIyamaJasso}, conditions were established for an algebra~$\Lambda$ to admit only finitely many isomorphism classes of indecomposable~$\tau$-rigid modules.
Such an algebra is called~\emph{$\tau$-tilting finite}.

An obvious sufficient condition for an algebra to be~$\tau$-tilting finite is for it to be representation-finite.  
This is not a necessary condition: for instance, if~$k$ is a field, then the algebra~$k\langle x,y \rangle/(x^2, y^2, xy, yx)$ is representation-infinite 
(since it is a string algebra in the sense of \cite{ButlerRingel} and admits at least one band, namely~$xy^{-1}$), but it is~$\tau$-tilting finite (since it is local).
Our aim in this note is to prove that, for a certain class of algebras called \emph{gentle algebras}, representation-finiteness and~$\tau$-tilting finiteness are equivalent conditions.

Gentle algebras form a subclass of the class of string algebras.
They enjoy a simple definition in terms of generators and relations: 
a gentle algebra is a finite-dimensional algebra isomorphic to a quotient of a path algebra of a finite quiver~$Q$ by an ideal~$I$ generated by paths of length two, 
satisfying the condition that for every vertex~$v$ of~$Q$, the minimal full subquiver with relation of~$\bar Q = (Q,I)$ containing~$v$ and all arrows attached to~$v$ is a full subquiver with relations of the one depicted below,
where dotted line indicate relations.
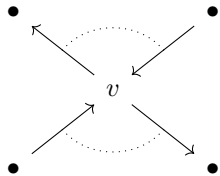
\begin{center}
\begin{tikzcd}
 \bullet & & \bullet\arrow[dl, ""{name=U}] \\
 & v \arrow[ul,""{name=V}] \arrow[dr, ""{name=X}] & \\
 \bullet\arrow[ur, ""{name=W}] && \bullet
 \arrow[dash, dotted, bend left=50, from=V, to=U]
 \arrow[dash, dotted, bend right=50, from=W, to=X]
\end{tikzcd}
\end{center}

Despite their simple definition, gentle algebras are encountered in many different contexts.
They were introduced in \cite{AssemSkowronski} in the study of iterated tilted algebras of type~$\widetilde{A}_m$, 
but have recently appeared in connection with dimer models \cite{Bocklandt, Broomhead}, enveloping algebras of some Lie algebras \cite{HuerfanoKhovanov}, 
cluster algebras and categories arising from triangulated surfaces \cite{Labardini,ABCP},~$m$-Calabi--Yau tilted algebras \cite{Garcia, Garcia2},
non-kissing complexes of grids and associated objects \cite{McConville, GarverMcConville,  PaluPilaudPlamondon, BrustleDouvilleMousavandThomasYildirim},
non-commutative nodal curves \cite{BurbanDrozd},
and partially wrapped Fukaya categories \cite{HaidenKatzarkovKontsevich, LekiliPolishchuk}.
Surface models have been introduced to study the category representations of a gentle algebra and associated categories \cite{BaurSimoes, OpperPlamondonSchroll, PaluPilaudPlamondon2}.

In this note, we prove the following theorem on gentle algebras.  It is proved for Schurian gentle algebras in \cite{DemonetIyamaPalu} (see also \cite{DemonetHDR}).

\begin{theorem}\label{theo::main}
 A gentle algebra is~$\tau$-tilting finite if and only if it is representation-finite.
\end{theorem}

The proof of the theorem uses the ``brick-$\tau$-tilting correspondence'' of \cite{DemonetIyamaJasso} (recalled in Section~\ref{sect::correspondence}),
and applies a reduction of any gentle algebra to two classes of examples (studied in Section~\ref{sect::twoExamples}).

\section*{Acknowledgements}
I am grateful to Laurent Demonet and Yann Palu for informing me of their work in preparation \cite{DemonetIyamaPalu}, 
and to Vincent Pilaud for several discussions on~$\tau$-tilting finite gentle algebras. 
I also thank Kaveh Mousavand for showing me families of examples of representation-infinite string algebras which are~$\tau$-tilting finite, 
for letting me know of his ongoing work on special biserial algebra regarding questions similar to those treated in this note,
and for pointing out a mistake in the proof of Proposition~\ref{prop::algebraIsTauInfinite} in an earlier version of this note.

\section*{Conventions and a note on terminology}
All algebras in this paper are finite-dimensional over a base field~$k$, which is arbitrary. 
We compose arrows in quivers from left to right: if~$1\xrightarrow{\alpha} 2 \xrightarrow{\beta} 3$ is a quiver, then~$\alpha\beta$ is a path, while~$\beta\alpha$ is not.
For any arrow~$\alpha$ of a quiver, we denote its source by~$s(\alpha)$ and its target by~$t(\alpha)$; we extend this notation naturally to the formal inverse~$\alpha^{-1}$ of~$\alpha$.

In order to keep this note short, the notions of strings, bands and string modules from \cite{ButlerRingel} will be used freely and without further introduction.

\section{The ``brick-$\tau$-tilting correspondence'' and two reduction results}\label{sect::correspondence}

We will be using two results, Corollaries~\ref{coro::firstReduction} and~\ref{coro::secondReduction},
allowing us to reduce the number of vertices of our bound quivers. 
They will both follow from the following consequence of the \emph{brick-$\tau$-tilting correspondence}.
Recall that a \emph{brick} is a module whose endomorphism algebra is a division algebra.

\begin{theorem}[Theorem 1.4 of \cite{DemonetIyamaJasso}]\label{theo::tau-tilting-bricks}
 A finite-dimensional algebra~$\Lambda$ is~$\tau$-tilting finite if and only if there are only finitely many isomorphism classes of~$\Lambda$-modules which are bricks.
\end{theorem}

\begin{corollary}\label{coro::functors}
 Let~$\Lambda$ and~$\Lambda'$ be two finite-dimensional algebra, 
 and assume that there exists a fully faithful functor~$F:\MOD \Lambda \to \MOD \Lambda'$.
 If~$\Lambda'$ is~$\tau$-tilting finite, then so is~$\Lambda$.
\end{corollary}
\begin{proof}
 If~$B$ is a brick over~$\Lambda$, then~$FB$ is a brick over~$\Lambda'$, since~$\End{\Lambda}(B)$ is isomorphic to~$\End{\Lambda'}(FB)$.
 Moreover, two bricks~$B$ and~$B'$ over~$\Lambda$ are isomorphic if and only if~$FB$ and~$FB'$ are isomorphic.
 Therefore, if~$\Lambda$ admits infinitely many bricks, then so does~$\Lambda'$.
 The result then follows from Theorem~\ref{theo::tau-tilting-bricks}.
\end{proof}

\begin{corollary}[First reduction, Theorem 5.12(d) of \cite{DIRRT}]\label{coro::firstReduction}
 If~$\Lambda$ is~$\tau$-tilting finite and~$I$ is an ideal in~$\Lambda$, then~$\Lambda/I$ is~$\tau$-tilting finite.
\end{corollary}
\begin{proof}
 Apply Corollary~\ref{coro::functors} to~$-\otimes_{\Lambda/I} \Lambda/I : \MOD \Lambda/I \longrightarrow \MOD \Lambda$.
\end{proof}

\begin{corollary}[Second reduction]\label{coro::secondReduction}
 If~$\Lambda$ is~$\tau$-tilting finite and~$e\in\Lambda$ is an idempotent, then~$e\Lambda e$ is~$\tau$-tilting finite.
\end{corollary}
\begin{proof}
 Apply Corollary~\ref{coro::functors} to~$\Hom{e\Lambda e}(\Lambda e, ?) : \MOD e\Lambda e \longrightarrow \MOD \Lambda$.
 Alternatively, apply \cite[Theorem 1.1]{PilaudPlamondonStella} for a proof which does not use the brick-$\tau$-tilting correspondence.
\end{proof}

\begin{remark}
 In practice, Corollary~\ref{coro::firstReduction} implies that erasing arrows or vertices from a~$\tau$-tilting finite algebra yields another~$\tau$-tilting finite algebra, 
 and Corollary~\ref{coro::secondReduction} implies that erasing an arrow and replacing the paths of length~$2$ that went through it by ``shortcut'' arrows also preserves~$\tau$-tilting finiteness.
\end{remark}

\section{Two classes of examples}\label{sect::twoExamples}
Our strategy to prove Theorem~\ref{theo::main} will be to reduce any gentle algebra to the two classes of examples presented in this section.

\begin{example}[Type~$\widetilde{A}_m$]\label{exam::AnTilde}
 Let~$Q$ be a quiver of type~$\widetilde{A}_m$, that is to say, an orientation of the following diagram with~$m+1$ vertices, where~$m\geq 1$.
 
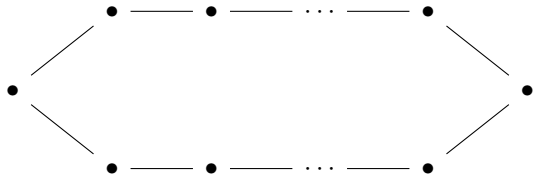
\begin{center}
  \begin{tikzcd}
   & \bullet\arrow[r, dash] & \bullet\arrow[r, dash] & \cdots & \bullet\ar[dr, dash]\arrow[l, dash] & \\
   \bullet\arrow[ur, dash]\arrow[dr, dash] & & & & & \bullet \\
   & \bullet\arrow[r, dash] & \bullet\arrow[r, dash] & \cdots & \bullet\ar[ur, dash]\arrow[l, dash] & \\
  \end{tikzcd}
 \end{center}
 The representation theory of the path algebra~$\Lambda = kQ$ is very well understood in this case, see for instance \cite[Section VIII.2]{AssemSimsonSkowronski}.
 In particular, this algebra is~$\tau$-tilting infinite.
\end{example}

\begin{example}\label{exam::other}
 The second class of examples that we will consider will be given by the path algebras of quivers~$Q$ defined by any orientation of the diagram below,
 
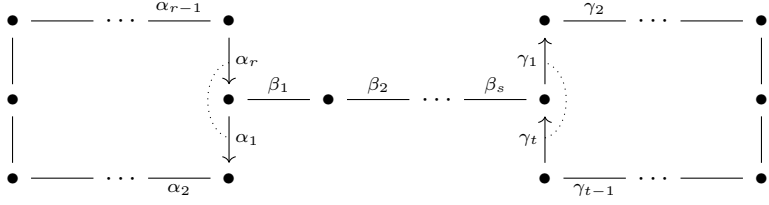
\begin{center}
  \begin{tikzcd}
    \bullet\arrow[r, dash]  & \cdots\arrow[r, dash, "\alpha_{r-1}"] & \bullet\ar[d, "\alpha_r", ""{name=A}]    &&& \bullet\arrow[r, dash, "\gamma_2"] & \cdots\arrow[r, dash] & \bullet\ar[d, dash] \\
   \bullet\arrow[u, dash]\arrow[d, dash]  & &  \bullet\arrow[d, "\alpha_1", ""{name=B}]                     \arrow[r, dash, "\beta_1"] & \bullet\arrow[r, dash,"\beta_2"] & \cdots\arrow[r, dash, "\beta_s"] & \bullet\arrow[u, "\gamma_1", ""{name=D}] & &\bullet\ar[d, dash] \\
    \bullet\arrow[r, dash]  & \cdots & \bullet\arrow[l, dash, "\alpha_2"]                           &&& \bullet\arrow[u, "\gamma_t",""{name=C}] & \cdots\arrow[l, dash, "\gamma_{t-1}"]\arrow[r, dash] & \bullet  \\
    
    \arrow[dash, dotted, from=A, to=B, bend right=75]
    \arrow[dash, dotted, from=C, to=D, bend right=75]
  \end{tikzcd}
 \end{center}
 modulo the relations~$\alpha_r\alpha_1$ and~$\gamma_t\gamma_1$ 
 (note that the orientations of~$\alpha_1,\alpha_r, \gamma_1$ and~$\gamma_t$ are imposed, while those of other arrows can be arbitrary).
 We allow~$r=1$ (in which case~$\alpha_1 = \alpha_r$ is a loop whose square vanishes); we allow the same for~$t$.
 We also allow~$s=0$; in this case, we require that the cycle on the left and the cycle on the right are not both oriented cycles, otherwise the algebra would be infinite-dimensional.
\end{example}

\begin{proposition}\label{prop::algebraIsTauInfinite}
 The algebras defined in Example~\ref{exam::other} are~$\tau$-tilting infinite.
\end{proposition}
\begin{proof}
 Let~$\Lambda$ be an algebra in the class defined in Example~\ref{exam::other}.
 
 We first deal with the case where~$s=0$.  In this case, let~$v$ be the vertex common to both cycles.
 Let~$e=1-e_v$.  Then the algebra~$e\Lambda e$ is of type~$\widetilde{A}_m$, so it is~$\tau$-tilting infinite.  By Corollary~\ref{coro::secondReduction},~$\Lambda$ is~$\tau$-tilting infinite.
 
 Assume, therefore, that~$s\geq 1$.  We will construct an infinite family of bricks for~$\Lambda$; 
 by the brick-$\tau$-tilting correspondence (see Theorem~\ref{theo::tau-tilting-bricks}), this will imply that~$\Lambda$ is~$\tau$-tilting infinite.
 
 Let~$b'=\alpha_1\alpha_2^{\delta_2}\cdots\alpha_{r-1}^{\delta_{r-1}}\alpha_r$ be the string corresponding to the cycle on the left,
 $b''=\gamma_1\gamma_2^{\zeta_2}\cdots\gamma_{t-1}^{\zeta_{t-1}}\gamma_t$ be the one corresponding to the cycle on the right,
 and $\omega=\beta_1^{\varepsilon_1}\cdots\beta_s^{\varepsilon_s}$ be the middle string followed from left to right,
 where the~$\delta_i, \varepsilon_i$ and~$\zeta_i$ are the appropriate signs.
 
 Define~$b = (b')^{\varepsilon_1}\omega(b'')^{\varepsilon_1}\omega^{-1}$. 
 We claim that the string module defined by~$b$ is a brick.
 To prove this, we need to prove that the only substring of~$b$ appearing both on top of and at the bottom of~$b$ is~$b$ itself, 
 so that the endomorphism ring of the string module is isomorphic to the base field~$k$ (using the descirption of all morphisms between string modules obtained in \cite{CrawleyBoevey}).
 Here, we say that a substring~$\sigma'$ of a string~$\sigma$ is \emph{on top of}~$\sigma$ if the arrows in~$\sigma$ adjacent to~$\sigma'$ are leaving~$\sigma'$,
 and that it is \emph{at the bottom of}~$\sigma$ if the arrows in~$\sigma$ adjacent to~$\sigma'$ are entering~$\sigma'$.
 
 Note first that the middle copy of~$\omega$ is neither on top of nor at the bottom of~$b$.
 Indeed, if~$\varepsilon_1 = 1$, then~$\omega$ is not at the bottom, since the first arrow of~$(b'')^{\varepsilon_1}$ is direct,
 and~$\omega$ is not on top, since the last arrow of~$(b')^{\varepsilon_1}$ is direct;
 if~$\varepsilon_1 = -1$, then~$\omega$ is not at the bottom, since the last arrow of~$(b')^{\varepsilon_1}$ is inverse,
 and~$\omega$ is not on top, since the last arrow of~$(b'')^{\varepsilon_1}$ is inverse.
 
 
 
 Next, let us deal with the substrings of length~$0$ of~$b$.  
 The starting point of~$b$ is on top if~$\varepsilon_1=1$ or at the bottom if~$\varepsilon_1=-1$.  It appears twice more in~$b$: at the end of~$b'$ and the end of~$b$.
 At the end of~$b'$ it is neither on top nor at the bottom, and at the end of~$b$ it is on top if~$\varepsilon_1=1$ or at the bottom if~$\varepsilon_1=-1$.
 Therefore this vertex does not occur both on top and at the bottom of~$b$.
 The other vertices appearing several times in~$b$ are the vertices of~$\omega$ and the starting/ending point of~$b''$.
 The former appear either twice at the top or twice at the bottom of~$b$.
 The latter cannot be both on top and at the bottom, since it appears in the middle of paths of length~$2$ of the following form:
 if~$\varepsilon_1 = 1$, then these paths are~$\beta_s^{\varepsilon_s}\gamma_1$ and~$\gamma_t \beta_s^{-\varepsilon_s}$, and 
 if~$\varepsilon_1 = -1$, then they are~$\beta_s^{\varepsilon_s}\gamma_t^{-1}$ and~$\gamma_1^{-1} \beta_s^{-\varepsilon_s}$. 
 In both cases the middle vertices is either on top of~$b$ or at the bottom of~$b$, but not both.
 Thus no substring of length~$0$ appears both at the top and the bottom of~$b$.
 
 Assume that there is a substring~$\rho$ of length at least~$1$, different from~$b$, which appears both on top and at the bottom of~$b$.
 Since the only arrows of~$b$ that are used twice are those of~$\omega$,~$\rho$ has to be a substring of~$\omega$ and of~$\omega^{-1}$.
 Since~$\omega$ does not go twice through the same vertex, the only substring both on top and at the bottom of~$\omega$ is~$\omega$ itself.  
 Hence~$\rho =\omega$.  But we saw above that the middle substring~$\omega$ is neither on top nor at the bottom of~$b$.  This is a contradiction.

 Thus the string module defined by~$b$ is a brick.
 
 Using the above arguments, one can also check that all powers of~$b$ define bricks as well.
 Thus~$\Lambda$ admits infinitely many pairwise non-isomorphic bricks, and by the brick-$\tau$-tilting correspondence (see Theorem~\ref{theo::tau-tilting-bricks}), it is~$\tau$-tilting infinite.
\end{proof}

\section{Proof of the main Theorem}
We now prove Theorem~\ref{theo::main}.  
Let~$\bar Q = (Q,I)$ be a gentle bound quiver.  
Assume that the algebra~$\Lambda = kQ/I$ is of infinite representation type.
Let us show that it is~$\tau$-tilting infinite.

By \cite{ButlerRingel}, there exists a band $b$ on~$\bar Q$. 
Our strategy will be to reduce to one of the two cases in the following lemma.

\begin{lemma}\label{lemm::reductionOnBands}
 The algebra~$\Lambda$ is~$\tau$-tilting infinite if~$\bar Q$ admits a band~$b$ satisfying one of the following conditions:
 \begin{enumerate}
  \item $b$ does not go through the same vertex twice (except for the starting and ending points of~$b$); or
  \item $b$ has the form~$b = b'\omega b''\omega^{-1}$, where
    \begin{itemize}
     \item ~$b'$ and~$b''$ are strings such that~$s(b') = t(b')$ and~$s(b'')=t(b'')$;
     \item ~$(b')^2$ and~$(b'')^2$ are not strings;
     \item ~$\omega$ is a possibly trivial string;
     \item none of~$b', b''$ and~$\omega$ go through the same vertex twice (except for the endpoints of~$b'$ and~$b''$);
     \item the only vertices that~$b', b''$ and~$\omega$ may have in common are their ending points.
    \end{itemize}
 \end{enumerate}
\end{lemma}
\begin{proof}
 Let~$I$ be the ideal generated by all arrows and vertices through which~$b$ does not go.  
 If we are in case (1), then~$\Lambda/I$ is isomorphic to the path algebra of a quiver of type~$\widetilde{A}_m$.
 These algebras are~$\tau$-tilting infinite, so the result follows from Corollary~\ref{coro::firstReduction}.
 If we are in case (2), then~$\Lambda/I$ is isomorphic to an algebra in the class defined in Example~\ref{exam::other}.
 By Proposition~\ref{prop::algebraIsTauInfinite}, these are~$\tau$-tilting infinite, so the result follows again from Corollary~\ref{coro::firstReduction}.
%
%
%
%
\end{proof}

To prove Theorem~\ref{theo::main}, it is therefore sufficient to show that a band as in Lemma~\ref{lemm::reductionOnBands} always exists.

Let~$b$ be a band of minimal length on~$\bar Q$.
If~$b$ does not go through the same vertex twice (except at its endpoints), then by Lemma~\ref{lemm::reductionOnBands} (1), the theorem is proved.

Assume, therefore, that there is a vertex~$u$ through which~$b$ passes twice.
Up to cyclic reordering of~$b$, we can assume that this vertex is the starting point of~$b$.
Up to choosing another such vertex~$u$, we can also assume that~$b=b'b''$, with~$b'$ and~$b''$ non-trivial strings starting and ending at~$u$ and such that~$b'$ does not go through the same vertex twice (except at its endpoints).
Note that, by minimality of~$b$, the string $b'$ cannot be a band; that is,~$(b')^2$ cannot be a string.
The same is true for~$b''$.

\begin{lemma}\label{lemm::noCommonVertices}
 Let~$b=b'b''$ be as above.  Then~$b'$ and~$b''$ cannot have a vertex in common apart from their starting and ending points. 
\end{lemma}
\begin{proof}
 Assume that~$b'$ and~$b''$ have another vertex in common, and let~$v$ be such a vertex.
Write~$b'=\alpha_1^{\delta_1} \cdots \alpha_r^{\delta_r}$ and~$b''=\beta_1^{\varepsilon_1} \cdots \beta_s^{\varepsilon_s}$, 
and assume that~$s(\alpha_j^{\delta_j}) = v = t(\beta_k^{\varepsilon_k})$, with~$i\neq 1,r$ and~$j\neq 1,s$.
By minimality of~$b$, $\alpha_j^{\delta_j} \alpha_{j+1}^{\delta_{j+1}} \cdots \alpha_r^{\delta_r}\beta_1^{\varepsilon_1} \cdots \beta_k^{\varepsilon_k}$ cannot be a band.
This implies that~$\delta_j = \varepsilon_k$ and that the composition of~$\alpha_j^{\delta_j}$ and~$\beta_k^{\varepsilon_k}$ is a relation in~$\bar Q$.
But then~$\alpha_{j-1}^{-\delta_{j-1}}\alpha_{j-2}^{-\delta_{j-2}}\cdots \alpha_{1}^{-\delta_{1}} \beta_1^{\varepsilon_1} \cdots \beta_k^{\varepsilon_k}$ is a band, 
since~$\beta_k^{\varepsilon_k}$ cannot be in a relation both with~$\alpha_j^{\delta_j}$ and with~$\alpha_{j-1}^{\delta_{j-1}}$.
This again contradicts the minimality of~$b$.
\end{proof}

Now, if~$b''$ does not go twice through the same vertex (except for its endpoints), then we are in case (2) of Lemma~\ref{lemm::reductionOnBands} with~$\omega$ trivial, and the theorem is proved.
Assume thus that~$b''$ does go twice through a vertex~$v$ outside its endpoints (or three times through~$v$ if it is the endpoint of~$b''$).
Up to choosing another vertex~$v$, we can assume that there is a substring~$b''' = \beta_j^{\varepsilon_j}\beta_{j+1}^{\varepsilon_{j+1}} \cdots \beta_{k}^{\varepsilon_k}$ of~$b''$
which starts and ends in~$v$, does not go through the same vertex twice outside its endpoints, and such that~$\omega:=\beta_1^{\varepsilon_1} \cdots \beta_{j-1}^{\varepsilon_{j-1}}$ does not go through the same vertex twice.

If~$\beta_{k}$ and~$\beta_{j-1}$ form a relation of~$\bar Q$, then~$\beta_{j}$ and~$\beta_{k}$ cannot form a relation, since there is at most one involving~$\beta_{k}$ and~$\beta_{j-1}, \beta_j$ are distinct.
Thus~$b'''$ is a band with no repeated vertices (except for its endpoints), and we have reduced to case (1) of Lemma~\ref{lemm::reductionOnBands}, proving the theorem.

If, on the other hand, $\beta_{k}$ and~$\beta_{j-1}$ do not form a relation, then~$b'\omega b'''\omega^{-1}$ is a band satisfying the conditions of case (2) of Lemma~\ref{lemm::reductionOnBands}.

This finishes the proof of Theorem~\ref{theo::main}.


\newcommand{\etalchar}[1]{$^{#1}$}

\end{document}